\documentclass[12pt]{article}
\usepackage[utf8]{inputenc}

\usepackage{amsmath}
\usepackage{amsthm}
\usepackage{amssymb}

\usepackage{hyperref}

\usepackage{cite}
\usepackage{color}

\usepackage{enumitem}
\setlist[enumerate]{label={\textnormal{(\roman*)}}}

\newtheorem{theorem}{Theorem}[section]
\newtheorem{lemma}[theorem]{Lemma}
\newtheorem{corollary}[theorem]{Corollary}

\theoremstyle{definition}

\newtheorem{conjecture}[theorem]{Conjecture}

\usepackage{tikz}
\tikzstyle{vertex}=[circle, draw, inner sep=0pt, minimum size=4pt,fill=black]
\newcommand{\vertex}{\node[vertex]}
\usetikzlibrary{arrows.meta}
\usetikzlibrary{decorations.markings}
\tikzset{->-/.style={decoration={
  markings,
  mark=at position #1 with {\arrow{Triangle}}},postaction={decorate}}}
\renewcommand{\tt}[1]{\textnormal{\texttt{{#1}}}}
\newcommand{\circular}[1]{\langle{#1}\rangle}

\title{Lengths of extremal square-free ternary words}
\author{Lucas~Mol and Narad~Rampersad\footnote{The work of Narad Rampersad is supported by the Natural Sciences and Engineering Research Council of Canada (NSERC), [funding reference number 2019-04111].}\\
\small Department of Mathematics and Statistics\\
\small The University of Winnipeg\\
\small 515 Portage Ave.\\
\small Winnipeg, MB, Canada\\
\small R3B 2E9\\
\small \{l.mol, n.rampersad\}@uwinnipeg.ca}
\date{}

\begin{document}

\maketitle

\begin{abstract}
\noindent
A square-free word $w$ over a fixed alphabet $\Sigma$ is \emph{extremal} if every word obtained from $w$ by inserting a single letter from $\Sigma$ (at any position) contains a square.  Grytczuk et al.\ recently introduced the concept of extremal square-free word, and demonstrated that there are arbitrarily long extremal square-free ternary words.  We find all lengths which admit an extremal square-free ternary word.  In particular, we show that there is an extremal square-free ternary word of every sufficiently large length.  We also solve the analogous problem for circular words.


\end{abstract}

\section{Introduction}

Throughout, we use standard definitions and notations from combinatorics on words (see~\cite{LothaireAlgebraic}).  The word $u$ is a \emph{factor} of the word $w$ if we can write $w=xuy$ for some (possibly empty) words $x,y$.  A word is \emph{square-free} if it contains no factor of the form $xx$, where $x$ is a nonempty word.
Early in the twentieth century, Norwegian mathematician Axel Thue demonstrated that one can construct arbitrarily long square-free words over a ternary alphabet (see~\cite{Berstel1995}). Thue's work is recognized as the beginning of the field of combinatorics on words~\cite{BerstelPerrin2007}.

Let $w$ be a word over a fixed alphabet $\Sigma$.  A \emph{left (right) extension} of $w$ is a word of the form $aw$ ($wa$, respectively), where $a\in \Sigma$.  We say that a square-free word $w$ is \emph{maximal} if both every left extension of $w$ contains a square, and every right extension of $w$ contains a square.  Bean, Ehrenfeucht, and McNulty~\cite{BEM1979} demonstrated that every square-free word over a fixed alphabet $\Sigma$ is a factor of a maximal square-free word over $\Sigma$.  (In fact, Bean, Ehrenfeucht, and McNulty established this result not only for square-free words, but for $k$th-power free words for every integer $k\geq 2$.)  A corollary is that there are arbitrarily long maximal square-free words over any alphabet of size at least $3$.

Grytczuk et al.~\cite{Grytczuk2019} recently introduced a variant of maximal square-free words, in which extensions not just at the beginning and the end, but at any point in the interior of the word, are considered.  Let $w$ be a word over a fixed alphabet $\Sigma$.  An \emph{extension} of $w$ is a word of the form $w'aw''$, where $a\in \Sigma$ and $w'w''=w.$  
We say that $w$ is \emph{extremal square-free} if $w$ is square-free, and there is no square-free extension of $w$.

Grytczuk et al.~\cite{Grytczuk2019} demonstrated that there are arbitrarily long extremal square-free ternary words.  In this article, we describe exactly those integers $n$ for which an extremal square-free ternary word of length $n$ exists.  In particular, we find that there is an extremal square-free ternary word of every sufficiently large length.  This confirms a conjecture of Jeffrey Shallit~\cite{ShallitConjecture}.

\begin{theorem}\label{AllLongSquareFree}
Let $n$ be a nonnegative integer.  Then there is an extremal square-free word of length $n$ over the alphabet $\Gamma=\{\tt{a},\tt{b},\tt{c}\}$ if and only if $n$ is in the set
\begin{align*}
\mathcal{A}&=\{25,41,48,50,63,71,72,77,79,81,83,84,85\}\cup\{m\colon\ m\geq 87\}.
\end{align*}
\end{theorem}

We also consider the analogous problem for \emph{circular words}.  The words $u$ and $v$ are \emph{conjugates} if there exist words $x$ and $y$ such that $u=xy$ and $v=yx$, i.e., if $u$ and $v$ are cyclic shifts of one another.  
Let $w\in\Sigma^*$.  The \emph{circular word} formed from $w$, denoted $\circular{w}$, is the set of all conjugates of $w$.  
For a set of words $L$, the word $u$ is a \emph{factor} of $L$ if $u$ is a factor of some word in $L$, and the set $L$ is \emph{square-free} if every word in $L$ is square-free.
In particular, the word $u$ is a factor of the circular word $\circular{w}$ if and only if $u$ is a factor of some conjugate of $w$, and the circular word $\circular{w}$ is square-free if and only if every conjugate of $w$ is square-free.  The following theorem was first proven by Currie~\cite{Currie2002}, and has since been reproven by several different methods~\cite{Shur2010,ClokieGabricShallit2019}.

\begin{theorem}[Currie~\cite{Currie2002}]\label{CurrieTernary}
For every integer $n\geq 18$, there is a square-free circular word of length $n$ over the alphabet \textnormal{$\{\tt{0},\tt{1},\tt{2}\}$}.
\end{theorem}

By Theorem~\ref{CurrieTernary} and a finite search, the only lengths which do not admit square-free ternary circular words are $5$, $7$, $9$, $10$, $14$, and $17$.

Let $w$ be a word over a fixed alphabet $\Sigma$.  An \emph{extension} of the circular word $\circular{w}$ is a circular word of the form $\circular{w'aw''}$, where $a\in\Sigma$ is a letter and $w=w'w''$.  The circular word $\circular{w}$ is \emph{extremal square-free} if $\circular{w}$ is square-free, and every extension of $\circular{w}$ contains a square.  We prove the following theorem concerning the attainable lengths of extremal square-free ternary circular words.  This can be regarded as a strengthening of Theorem~\ref{CurrieTernary}.

\begin{theorem}\label{LongCircular}
Let $n$ be a nonnegative integer.  Then there is an extremal square-free circular word of length $n$ over the alphabet $\Gamma=\{\tt{a},\tt{b},\tt{c}\}$ if and only if $n$ is in the set
\begin{align*}
\mathcal{B}&=\{4,6,8,13,15,16,18,20,21,22,23,24,28,30,32,33,34,35,36\}\\
&\hspace{1cm}\cup\{m\colon\ m\geq 38\}.
\end{align*}
\end{theorem}

The layout of the remainder of the article is as follows.  In Section~\ref{Prelims}, we present some preliminaries which are used to prove both of our main results.  In Section~\ref{LinearSection}, we prove Theorem~\ref{AllLongSquareFree}.  In Section~\ref{CircularSection}, we prove Theorem~\ref{LongCircular}.  We conclude with a discussion of some related problems.

\section{Preliminaries}\label{Prelims}

We will need the following well-known lemma, attributed to Sylvester.  See~\cite[Section~2.1]{FrobeniusProblemBook} for several different proofs.

\begin{lemma}\label{PostageStamp}
Let $p$ and $q$ be relatively prime positive integers.  For every integer $n\geq (p-1)(q-1)$, there exist nonnegative integers $a$ and $b$ such that $n=ap+bq$.
\end{lemma}

We will also need the following corollary of Lemma~\ref{PostageStamp}.

\begin{corollary}\label{EvenPostage}
Let $p$ and $q$ be relatively prime positive integers, exactly one of which is even.  For every integer $n\geq pq+(p-1)(q-1)$, there exist nonnegative integers $a$ and $b$ such that $n=ap+bq$, and the sum $a+b$ is even.
\end{corollary}

\begin{proof}
Suppose without loss of generality that $p$ is even and $q$ is odd.  Let $n\geq pq+(p-1)(q-1)$.  Then we have $n-pq\geq (p-1)(q-1)$.  By Lemma~\ref{PostageStamp}, there are nonnegative integers $\alpha$ and $\beta$ such that $n-pq=\alpha p+\beta q$.  If $\alpha+\beta$ is odd, then we can write $n=(\alpha+q)p+\beta q$, and $\alpha+q+\beta$ is even.  If $\alpha+\beta$ is even, then we can write $n=\alpha p+(\beta+p)q$, and $\alpha+\beta+p$ is even.
\end{proof}

Next, we prove a theorem which essentially extends a result of Grytczuk et al.~\cite[Theorem 2]{Grytczuk2019} from a morphism to a multi-valued substitution.  We note that many results similar to~\cite[Theorem 2]{Grytczuk2019} have appeared before in the literature (see~\cite[Section 4.2.5]{RampersadShallitChapter} for a summary).  However, most of these results give conditions on a morphism $f:\Sigma^*\rightarrow \Delta^*$ which guarantee that $f(w)$ is square-free for \emph{every} square-free word $w\in\Sigma^*$.  By contrast, the result of Grytczuk et al.\ gives conditions on a morphism $f:\Sigma^*\rightarrow\Delta^*$ and a square-free word $w\in \Sigma^*$ which guarantee that the word $f(w)$ is square-free, i.e., the conditions depend explicitly on the word $w$.  We note that arguments similar to those used in the proof of the following theorem have appeared before (see the proof of~\cite[Lemma 8]{KarhumakiShallit2004}, for example), but the entire proof is included for completeness.

\begin{theorem}\label{SquareFreeSubstitution}
Let $f:\Sigma^*\rightarrow 2^{\Delta^*}$ be a substitution, and let $u\in \Sigma^*$ be a square-free word.  Then the set $f(u)$ is square-free if all of the following conditions are satisfied:
\begin{enumerate}[label=\textnormal{(\Roman*)}]
    \item For every factor $v$ of $u$ of length at most $3$, the set $f(v)$ is square-free.\label{short}
    \item For every $a,b,c\in \Sigma$, and every $A\in f(a)$, $B\in f(b)$, and $C\in f(c)$:\label{forall}
    \begin{enumerate}[label=\textnormal{(\roman*)}]
            \item If $A$ is a factor of $B$, then $a=b$ and $A=B$.\label{infix}
        \item If $AB=pCs$ for some words $p,s\in\Delta^*$, then $p=\varepsilon$ or $s=\varepsilon$.\label{factor}
        \item If $A=A'A''$, $B=B'B''$, and $C=A'B''$, then $c=a$ or $c=b$.\label{final}
    \end{enumerate} 
\end{enumerate}
\end{theorem}

\begin{proof}
Suppose towards a contradiction that conditions \ref{short} and \ref{forall} are satisfied, but that some word in $f(u)$ contains a square.  Let $w=a_1a_2\cdots a_n$ be a minimal factor of $u$ such that some word $W=A_1A_2\cdots A_n$ contains a square, where $A_i\in f(a_i)$ for all $i\in\{1,\dots,n\}$.  Write $W=XYYZ$.  By the minimality of $w$, the word $X$ must be a proper prefix of $A_1$, and the word $Z$ must be a proper suffix of $A_n$.  By condition \ref{short}, we have $n\geq 4$. 

Suppose that $n=4$.  Then we can write
\[
W=A_1'A_1''A_2'A_2''A_3A_4'A_4'',
\]
where $A_i=A_i'A_i''$ for all $i\in\{1,2,4\}$, and $Y=A_1''A_2'=A_2''A_3A_4'$, or we can write
\[
W=A_1'A_1''A_2A_3'A_3''A_4'A_4'',
\]
where $A_i=A_i'A_i''$ for all $i\in\{1,3,4\}$, and $Y=A_1''A_2A_3'=A_3''A_4'$.  Assume the former. (The latter is handled similarly.) 
Then the word $A_3$ is a factor of $A_1''A_2'$, and hence of $A_1A_2$.  By Condition \ref{forall}\ref{factor}, we must have either $A_1'=A_2''=\varepsilon$, or $A_4'=\varepsilon$.  Since $A_4'=\varepsilon$ is impossible by the minimality of $w$, we must have $A_1'=A_2''=\varepsilon$.  Then $A_1A_2=A_3A_4'$, and hence either $A_1$ is a factor of $A_3$, or vice versa.  By Condition \ref{forall}\ref{infix}, we must have $a_1=a_3$ and $A_1=A_3$.  It follows that $A_2$ is a factor of $A_4'$, and hence $a_2=a_4$.  But then $w$ contains the square $(a_1a_2)^2$, an impossibility.

So we may assume that $n\geq 5$.  For some $j\in\{2,\dots,n-1\}$, we can write
\begin{align}\label{FormOfW}
W=A_1A_2\cdots A_n=A'_1A''_1A_2\cdots A_{j-1}A'_jA''_jA_{j+1}\cdots A_{n-1}A'_nA''_n,
\end{align}
where $A_i=A'_iA''_i$ for all $i\in\{1,j,n\}$, $X=A'_1$, $Z=A''_n$, and 
\begin{align*}
Y&=A''_1A_2\cdots A_{j-1}A'_j=A''_jA_{j+1}\cdots A_{n-1}A'_n.
\end{align*}
By the minimality of $w$, we must have $|A_1''|,|A_n'|>0$, and we may assume without loss of generality that $|A_j''|>0$.

Suppose that $|A_1''|>|A_j''|$.  If $j=2$, then $A_1''A_2'=A_2''A_3\cdots A_{n-1} A_n'$.  But since $n\geq 5$, we see that either $A_3$ must be a proper factor of $A_1''$, or $A_4$ must be a proper factor of $A_2'$.  By condition \ref{forall}\ref{infix}, this is impossible.  So we may assume that $j>2$.  By condition \ref{forall}\ref{infix}, we have that $A_2$ is not a factor of $A_{j+1}$, so $A_{j+1}$ must be a factor of $A_1''A_2$.  In particular, this implies that $j+1<n$.  Write $A_{j+2}=A'_{j+2}A''_{j+2}$ so that $A_1'A_2=A''_jA_{j+1}A'_{j+2}$.  By condition (II)(ii), we must have either $|A_j''|=0$ or $|A_{j+2}'|=0$.  Since $|A_j''|>0$ by assumption, we must have $|A_{j+2}'|=0$.  By condition \ref{forall}\ref{infix}, we must have $A_2=A_{j+1}$, and hence $A_1''=A_j''$.  This contradicts the assumption that $|A_1''|>|A_j''|$.

Now suppose that $|A_j''|>|A_1''|$.  If $j=n-1$, then $A_1''A_2\cdots A_{n-2}A_{n-1}'=A_{n-1}''A_n'$.  But since $n\geq 5$, we see that either $A_2$ must be a proper factor of $A_{n-1}''$, or $A_3$ must be a proper factor of $A_n'$.  By condition \ref{forall}\ref{infix}, this is impossible.  So we may assume that $j<n-1$.  By condition \ref{forall}\ref{infix}, we have that $A_{j+1}$ is not a factor of $A_2$, so $A_2$ must be a factor of $A_j''A_{j+1}$.  Write $A_3=A_3'A_3''$, where $A_1''A_2A_3'=A_j''A_{j+1}$.  By condition \ref{forall}\ref{factor}, we must have $|A_1''|=0$ or $|A_3'|=0$.  Since $|A_1''|>0$ by assumption, we must have $|A_3'|=0$.  By condition \ref{forall}\ref{infix}, we must have $A_2=A_{j+1}$, and hence $A_1''=A_j''$.  This contradicts the assumption that $|A_j''|>|A_1''|$.

So we may assume that $|A_1''|=|A_j''|$, and hence $A_1''=A_j''$.  Then either $A_2$ is a factor of $A_{j+1}$, or vice versa.  By condition~\ref{forall}\ref{infix}, we conclude that $a_2=a_{j+1}$ and $A_2=A_{j+1}$.  Applying this argument repeatedly, we find $a_2a_3\cdots a_{j-1}=a_{j+1}a_{j+2}\cdots a_{n-1}$, and $A_2A_3\cdots A_{j-1}=A_{j+1}A_{j+2}\cdots A_{n-1}$.  Finally, we see that $A_j'=A_n'$.  But then $A_j=A_j'A_j''=A_n'A_1''$.  By condition~\ref{forall}\ref{final}, we conclude that $a_j=a_1$ or $a_j=a_n$.  But then the word $u$ contains either the square $(a_1a_2\cdots a_{j-1})^2$ or the square $(a_2a_3\cdots a_j)^2$, respectively.
\end{proof}

\section{Extremal square-free words}\label{LinearSection}

In this section, we prove Theorem~\ref{AllLongSquareFree}.  We first summarize the method of Grytczuk et al.~\cite{Grytczuk2019} used to construct arbitrarily long extremal square-free ternary words.  The proof of Theorem~\ref{AllLongSquareFree} is obtained by a similar method.  We first introduce some notation and terminology.

Let $w$ be a word over a fixed alphabet $\Sigma$.  We say that $w$ is \emph{nearly extremal square-free} if $w$ is square-free, and there are only two square-free extensions of $w$; one left extension, and one right extension.  We say that $w$ is \emph{left (right) extremal square-free} if every square-free extension of $w$ is a right (left, respectively) extension.

Let $\mathbb{S}_{\Gamma}$ denote the symmetric group on the alphabet $\Gamma=\{\tt{a},\tt{b},\tt{c}\}$.  We represent the permutations of $\mathbb{S}_{\Gamma}$ using cycle notation, but for ease of notation, we omit the commas.  We denote the identity permutation by the empty cycle $()$.  We treat every member of $\mathbb{S}_\Gamma$ as both a permutation, and as a letter, depending on context.  For every permutation $\pi\in\mathbb{S}_\Gamma$, we let $\Tilde{\pi}$ be another letter, which we refer to as the \emph{mirror image} of $\pi$.  Let $\Tilde{\mathbb{S}}_\Gamma=\{\Tilde{\pi}\colon\ \pi\in\mathbb{S}_\Gamma\}$.  Since every permutation of $\Gamma$ can be written as either a single nontrivial cycle or the empty cycle, the parentheses serve as delimiters for letters in words over the alphabet $\mathbb{S}_\Gamma\cup \Tilde{\mathbb{S}}_\Gamma$.

Let $D$ be the digraph with vertex set $V(D)=\mathbb{S}_\Gamma\cup\Tilde{\mathbb{S}}_\Gamma$ that is shown in Figure~\ref{GrytczukDigraph}.  Let 
\[
N=\tt{abacbabcabacbcacbabcabacabcbabcabacbcabcb}.
\]
It is easily checked by computer that the word $N$ is nearly extremal square-free.  Now for every permutation $\pi\in \mathbb{S}_\Gamma$, let $N_{\pi}$ denote the word obtained by permuting the letters of $N$ by $\pi$, and let $N_{\Tilde{\pi}}$ denote the reversal of $N_{\pi}$.  Define the morphism $f:V(D)^*\rightarrow \Gamma^*$ by $f(x)=N_x$ for all $x\in V(D)$. 

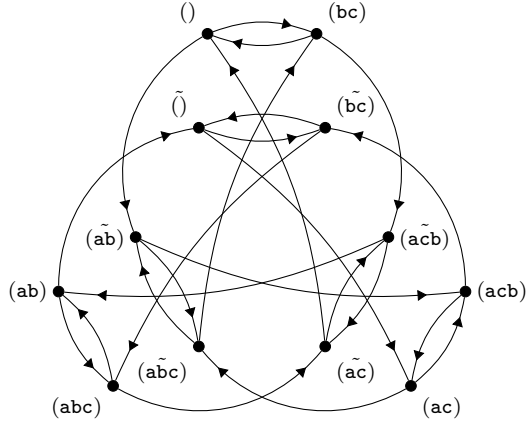
\begin{figure}
    \centering
    \begin{tikzpicture}[scale=1.4]
    \vertex (6) at (120:1.2) {};
    \node[above left] at (120:1.2) {\scriptsize $\Tilde{()}$};
    \vertex (7) at (180:1.2) {};
    \node[left] at (180:1.2) {\scriptsize $\Tilde{(\tt{ab})}$};
    \vertex (8) at (300:1.2) {};
    \node[below right] at (300:1.2) {\scriptsize $\Tilde{(\tt{ac})}$};
    \vertex (9) at (60:1.2) {};
    \node[above right] at (60:1.2) {\scriptsize $\Tilde{(\tt{bc})}$};
    \vertex (10) at (240:1.2) {};
    \node[below left] at (240:1.2) {\scriptsize $\Tilde{(\tt{abc})}$};
    \vertex (11) at (0:1.2) {};
    \node[right] at (0:1.2) {\scriptsize $\Tilde{(\tt{acb})}$};
    \vertex (0) at (105:2) {};
    \node[above left] at (105:2) {\scriptsize $()$};
    \vertex (1) at (195:2) {};
    \node[left] at (195:2) {\scriptsize $(\tt{ab})$};
    \vertex (2) at (315:2) {};
    \node[below right] at (315:2) {\scriptsize $(\tt{ac})$};
    \vertex (3) at (75:2) {};
    \node[above right] at (75:2) {\scriptsize $(\tt{bc})$};
    \vertex (4) at (225:2) {};
    \node[below left] at (225:2) {\scriptsize $(\tt{abc})$};
    \vertex (5) at (345:2) {};
    \node[right] at (345:2) {\scriptsize $(\tt{acb})$};
    \path
    (0) edge[->-=0.9,bend right=40] (7)
    (8) edge[->-=0.9,bend right=15] (0)
    (1) edge[->-=0.9,bend left=40] (6)
    (11) edge[->-=0.9,bend left=15] (1)
    (6) edge[->-=0.9,bend left=15] (2)
    (2) edge[->-=0.9,bend left=40] (10)
    (10) edge[->-=0.9,bend left=15] (3)
    (3) edge[->-=0.9,bend left=40] (11)
    (4) edge[->-=0.9,bend right=40] (8)
    (9) edge[->-=0.9,bend right=15] (4)
    (7) edge[->-=0.9,bend right=15] (5)
    (5) edge[->-=0.9,bend right=40] (9);
    \path
    (0) edge[->-=0.8, bend left=20] (3)
    (1) edge[->-=0.8, bend right=20] (4)
    (2) edge[->-=0.8, bend right=20] (5)
    (9) edge[->-=0.8, bend right=20] (6)
    (10) edge[->-=0.8, bend left=20] (7)
    (11) edge[->-=0.8, bend left=20] (8)
    (3) edge[->-=0.8, bend left=20] (0)
    (4) edge[->-=0.8, bend right=20] (1)
    (5) edge[->-=0.8, bend right=20] (2)
    (6) edge[->-=0.8, bend right=20] (9)
    (7) edge[->-=0.8, bend left=20] (10)
    (8) edge[->-=0.8, bend left=20] (11)
    ;
    \end{tikzpicture}
    \caption{The Digraph $D$.}
    \label{GrytczukDigraph}
\end{figure}

First of all, Grytczuk et al.\ show that if $w$ is a square-free walk in $D$ (where walks are treated as words over the vertex set), then the word $f(w)$ is square-free.  Next, they show that there are arbitrarily long square-free walks in $D$ that begin and end at the vertex $()$.  It follows that there are arbitrarily long nearly extremal square-free words over $\Gamma$ that have $N$ as both a prefix and a suffix.  Finally, Grytczuk et al.\ provide two short words that can be added to the beginning and the end of any such word to form an extremal square-free word.

In order to prove Theorem~\ref{AllLongSquareFree}, we replace the morphism $f$ with a multi-valued substitution $\delta$.  Using Lemma~\ref{PostageStamp}, we can then construct nearly extremal square-free words of every sufficiently large length.  Finally, we find words of a single fixed length that can be added to the beginning and end of every such nearly extremal square-free word to form an extremal square-free word.  This guarantees the existence of extremal square-free words of every sufficiently large length, and the smaller lengths are handled computationally.

Let
\begin{align*}
P&=\tt{abacbcabcbacabacbcabcbabcacbcabcbacabacbcabcbacbc},\\
Q&=\tt{abacbabcacbacabacbcacbacabcbabcabacbcabcb},\\
R&=\tt{abacabcacbacabcbabcacbacabacbcacbacabcbabcabacbcabcb},
\text{ and}\\
S&=\tt{acabacbabcacbacabcbacbcabacbabcacbacabcbabcacbaca}.
\end{align*}
We note that the words $P$ and $S$ have length $49$, the word $Q$ has length $41$, and the word $R$ has length $52$.  The word $R$ can be obtained from $Q$ by inserting the word \tt{abcacbacabc} after the $4$th letter.  By computer check, the words $Q$ and $R$ are both nearly extremal square-free, the words $PQ$ and $PR$ are left extremal square-free, and the words $QS$ and $RS$ are right extremal square-free.  Note that the word $Q$ cannot be obtained from the word $N$ (used by Grytczuk et al.) by permutation of the alphabet and/or reversal, i.e., for every $x\in V(D)$, we have $N_x\neq Q$.

For every letter $x\in V(D)$, let $p_x$ and $s_x$ be two new letters.  Let $\hat{D}$ be the graph obtained from $D$ by adding, for every $x\in V(D)$, the vertices $p_x$ and $s_x$, as well as arcs from $p_x$ to $x$ and from $x$ to $s_x$.  For every $x\in V(D)$, define the words $P_x$, $Q_x$, $R_x$, and $S_x$ analogously to $N_x$.  Now define the substitution $\delta:V(\hat{D})^*\rightarrow 2^{\Gamma^*}$ by 
\begin{itemize}
    \item $\delta(x)=\{Q_{x},R_{x}\}$ for all $x\in \mathbb{S}_\Gamma\cup\Tilde{\mathbb{S}}_\Gamma$;
    \item $\delta(p_x)=\{P_x\}$ and $\delta(s_x)=\{S_x\}$ for all $x\in \mathbb{S}_\Gamma$; and
    \item $\delta(p_x)=\{S_x\}$ and $\delta(s_x)=\{P_x\}$ for all $x\in \Tilde{\mathbb{S}}_\Gamma$.
\end{itemize}
Since $Q$ and $R$ are nearly extremal square-free, it follows immediately that $Q_x$ and $R_x$ are nearly extremal square-free for every $x\in V(D)$.  We also have the following fact.

\begin{lemma}\label{Endpoints}
For every $x\in V(D)$, every word in the set $\delta(p_xx)$ is left extremal square-free, and every word in the set $\delta(xs_x)$ is right extremal square-free.
\end{lemma}

\begin{proof}
Let $x\in V(D)$.  We show that every word in the set $\delta(p_xx)$ is left extremal square-free.  The proof that every word in the set $\delta(xs_x)$ is right extremal square-free is similar.  If $x\in \mathbb{S}_\Gamma$, then $\delta(p_xx)=\{P_xQ_x,P_xR_x\}$.   Since the words $PQ$ and $PR$ are left extremal square-free, so are $P_xQ_x$ and $P_xR_x$.  On the other hand, if $x\in \Tilde{\mathbb{S}}_\Gamma$, then write $x=\Tilde{\pi}$. Then $\delta(p_xx)=\{S_{\Tilde{\pi}}Q_{\Tilde{\pi}},S_{\Tilde{\pi}}R_{\Tilde{\pi}}\}$.  Note that $S_{\Tilde{\pi}}Q_{\Tilde{\pi}}$ is the reversal of the word $Q_{\pi}S_{\pi}$.  Since $QS$ is right extremal square-free, so is the word $Q_{\pi}S_{\pi}$.  It follows that the word $S_{\Tilde{\pi}}Q_{\Tilde{\pi}}$ is left extremal square-free.  The proof that $S_{\Tilde{\pi}}R_{\Tilde{\pi}}$ is left extremal square-free is analogous.
\end{proof}

Using Theorem~\ref{SquareFreeSubstitution}, the next lemma can be verified by a computer check.  (We check condition (I) for all square-free walks of length $3$ in $\hat{D}$.)

\begin{lemma}\label{SquareFreeHatD}
If $w$ is a square-free walk in the digraph $\hat{D}$, then the set $\delta(w)$ is square-free.
\end{lemma}

Together, Lemma~\ref{Endpoints} and Lemma~\ref{SquareFreeHatD} yield the following corollary.

\begin{corollary}\label{ExtremalSquareFreeWalks}
Let $w$ be a square-free walk in $D$ of length at least $2$, and write $w=xw'y$, where $x,y\in V(D)$.  Then every word in the set $\delta(p_xws_y)$ is extremal square-free.
\end{corollary}

\begin{proof}
Since $w$ is a square-free walk in $D$ (and not $\hat{D}$), it contains neither $p_x$ nor $s_y$, and hence the walk $p_xws_y$ is square-free.  By Lemma~\ref{SquareFreeHatD}, the set $\delta(p_xws_y)$ is square-free.  Note that for every $z\in V(D)$, both words in the set $\delta(z)$, namely $Q_z$ and $R_z$, are nearly extremal square-free.  It follows easily that every word in the set $\delta(w)$ is nearly extremal square-free.  By Lemma~\ref{Endpoints}, every word in $\delta(p_xx)$ is left extremal square-free, and every word in $\delta(ys_y)$ is right extremal square-free.  It follows that every word in the set $\delta(p_xws_y)$ is extremal square-free.
\end{proof}

We are now ready to prove Theorem~\ref{AllLongSquareFree}.

\begin{proof}[Proof of Theorem~\ref{AllLongSquareFree}]
If $n\not\in\mathcal{A}$, then we used a standard backtracking algorithm to show that there is no extremal square-free ternary word of length $n$.

Suppose otherwise that $n\in\mathcal{A}$.  Suppose first that $n\geq 2138$.  Then $n-98\geq 2040$, and by Lemma~\ref{PostageStamp}, we can write $n-98=41a+52b$ for some nonnegative integers $a$ and $b$.  Let $w$ be a square-free walk in $D$ of length $a+b$, and write $w=xw'y$ for some $x,y\in V(D)$.  Since $\delta(z)$ contains a word of length $41$ and a word of length $52$ for every $z\in V(D)$, it is evident that there is a word $W$ of length $n-98$ in $\delta(w)$.  Now the unique word $W'$ in $\delta(p_x)W\delta(s_y)$ has length $n$, and by Corollary~\ref{ExtremalSquareFreeWalks}, the word $W'$ is extremal square-free.

So we may assume that $n<2138$.  In this case, we found an extremal square-free circular word of length $n$ by computer search.
\end{proof}

\section{Extremal square-free circular words}\label{CircularSection}

In this section, we prove Theorem~\ref{LongCircular}.  A \emph{circumnavigation} of a circular word $\circular{w}$ is a linear word of the form $ava$, where $a$ is a letter, and $av$ is a conjugate of $w$.  We begin with an elementary lemma about the circumnavigations of square-free circular words.

\begin{lemma}\label{CircumnavigationsSquareFree}
Let $w\in \Sigma^*$ be a word of length at least $2$.  If $\circular{w}$ is square-free, then every circumnavigation of $\circular{w}$ is square-free.
\end{lemma}

\begin{proof}
Let $u$ be a circumnavigation of $\circular{w}$.  Since $|w|\geq 2$, we have $|u|\geq 3$.  Suppose towards a contradiction that $u$ contains a square.  Write $u=ava$, where $a$ is a letter, and $av$ is a conjugate of $w$.  Then $va$ is also a conjugate of $w$.  Since $\circular{w}$ is square-free, neither conjugate $av$ nor $va$ contains a square.  Hence, it must be the case that $u=xx$ for some word $x$.  Evidently, the word $x$ begins and ends in $a$, and has length at least $2$, since $|u|\geq 3$.  Write $x=ax'a$ for some word $x'$.  Then $u=ax'aax'a$, and we conclude that $\circular{w}$ contains the square $aa$, a contradiction.
\end{proof}

With $D$ the digraph shown in Figure~\ref{GrytczukDigraph}, define the substitution $h\colon \{\tt{0},\tt{1},\tt{2}\}^*\rightarrow 2^{V(D)^*}$ by
\begin{align*}
    \tt{0}&\rightarrow \left\{()\Tilde{(\tt{ab})}(\tt{acb})(\tt{ac})\Tilde{(\tt{abc})}(\tt{bc})\right\}\\
    \tt{1}&\rightarrow \left\{()\Tilde{(\tt{ab})}\Tilde{(\tt{abc})}(\tt{bc})\Tilde{(\tt{acb})}\Tilde{(\tt{ac})}\right\}\\
    \tt{2}&\rightarrow \left\{()\Tilde{(\tt{ab})}(\tt{acb})\Tilde{(\tt{bc})}(\tt{abc})\Tilde{(\tt{ac})},()\Tilde{(\tt{ab})}(\tt{acb})\Tilde{(\tt{bc})}(\tt{abc})(\tt{ab})(\tt{abc})\Tilde{(\tt{ac})}\right\}.
\end{align*}
Let $w=w_0w_1\cdots w_{n-1}$ be a word over the alphabet $V(D)$, where the $w_i$'s are letters.  We say that the circular word $\circular{w}$ is \emph{walkable} in $D$ if every conjugate of $w$ is a valid walk in $D$.  Equivalently, the circular word $\circular{w}$ is walkable in $D$ if there is an arc from $w_i$ to $w_{i+1}$ for every $i\in\{0,1,\hdots,n-1\}$, with indices taken modulo $n$.

\begin{lemma}\label{CircularSquareFreeMorphism}
Let $\circular{w}$ be a square-free circular word of length at least $2$ over the alphabet \textnormal{$\{\tt{0},\tt{1},\tt{2}\}$}, and let $W\in h(w)$.   Then the circular word $\circular{W}$ is square-free and walkable in the digraph $D$.
\end{lemma}

\begin{proof}
We first show that $\circular{W}$ is square-free. Let $U$ be a conjugate of $W$.  Then $U$ is a factor of the set $h(u)$ for some circumnavigation $u$ of $w$.  By Lemma~\ref{CircumnavigationsSquareFree}, the circumnavigation $u$ is square-free.  Using Theorem~\ref{SquareFreeSubstitution}, we verify by computer that $h(u)$ is square-free.  (We check condition (I) for every square-free word $v\in\{\tt{0},\tt{1},\tt{2}\}^*$ of length $3$.) We conclude that the word $U$ is square-free.  Since $U$ was an arbitrary conjugate of $W$, we conclude that the circular word $\circular{W}$ is square-free.

It remains to show that $\circular{W}$ is walkable in $D$.  Note that every word $A\in h(\{0,1,2\})$ begins in the identity permutation $()$.  So it suffices to check that for all $A\in h(\{0,1,2\})$, the word $A()$ is walkable in $D$, and this is easily done by inspection.
\end{proof}

\begin{lemma}\label{EvenLengthWalks}
For every even positive integer $n$, there is a square-free circular word of length $n$ that is walkable in the digraph $D$.
\end{lemma}

\begin{proof}
Let $n$ be an even positive integer.  First suppose that $n\geq 6\cdot 18=108$.  Then we may write $n=6m+r$ for some $m\geq 18$ and $r\in\{0,2,4\}$.  By Theorem~\ref{CurrieTernary}, there is a circular square-free word $\circular{u}$ of length $m$ over the alphabet $\{\tt{0},\tt{1},\tt{2}\}$.  Since every square-free word on $\{\tt{0},\tt{1}\}^*$ has length at most $3$, we must have $|u|_{\tt{2}}\geq 2$.  Note that for every $a\in\{\tt{0},\tt{1},\tt{2}\}$, there is a word in $h(a)$ of length $6$.  Further, the set $h(\tt{2})$ contains a word of length $8$.  Thus, there is a word $U\in h(u)$ of length $n=6m+r$, obtained by using the word of length $8$ in $h(\tt{2})$ exactly $0$, $1$, or $2$ times (for $r$ equal to $0$, $2$, or $4$, respectively).  By Lemma~\ref{CircularSquareFreeMorphism}, the circular word $\circular{U}$ is square-free, and is walkable in the digraph $D$.

Now we may assume that $n<108$, and in this case we verify the statement by means of a computer search.
\end{proof}

Let 
\begin{align*}
Q'&=\footnotesize{\tt{abacabcacbacabcbabcacbacabacbcacbacabcbacbcabacbabcabacbcabcb}, }\text{ and}\\
R'&=\footnotesize{\tt{abacabcacbacabcbabcacbacabacbcacbacabcbabcacbcabacbabcabacbcabcb}.
}\end{align*}
Note that $Q'$ has length $61$, and $R'$ has length $64$.  The word $Q'$ can be obtained from the word $R$ by adding the factor \tt{cbcabacba} after the $40$th letter, and the word $R'$ can be obtained from the word $Q'$ by adding the factor \tt{bca} after the $40$th letter.
For every $x\in \mathbb{S}_\Gamma\cup\Tilde{\mathbb{S}}_\Gamma$, define $Q'_x$ and $R'_x$ analogously to $N_x$.
Define the substitution $\delta':V(D)^*\rightarrow 2^{\Gamma^*}$ by $\delta'(x)=\{Q_{x},R_{x},Q'_{x},R'_{x}\}$ for all $x\in \mathbb{S}_\Gamma\cup\Tilde{\mathbb{S}}_\Gamma$.  So for every $x\in V(D)$, the set $\delta'(x)$ contains a word of length $m$ for all $m\in\{41,52,61,64\}$.  The proof of the following lemma is analogous to the first paragraph of the proof of Lemma~\ref{CircularSquareFreeMorphism}.

\begin{lemma}\label{deltaCircular}
Let $\circular{w}$ be a square-free circular word that is walkable in the digraph $D$, and let $W\in \delta'(w)$.  Then the circular word $\circular{W}$ is square-free. 
\end{lemma}

We are now ready to prove Theorem~\ref{LongCircular}.

\begin{proof}[Proof of Theorem~\ref{LongCircular}]
If $n\not\in \mathcal{B}$, then we used a standard backtracking algorithm to show that there is no extremal square-free ternary circular word of length $n$.  

Suppose otherwise that $n\in \mathcal{B}$.  First suppose that $n\geq 470$.  Then there are nonnegative integers $a$, $b$, $c$, and $d$ such that $41a+52b+61c+64d=n$, and the sum $a+b+c+d$ is even.  (For $n\geq 4172$, this claim follows from Lemma~\ref{EvenPostage}, and we verified the remaining cases by computer.)  By Lemma~\ref{EvenLengthWalks}, there is a square-free circular word $\circular{w}$ of length $a+b+c+d$ that is walkable in the digraph $D$.  Evidently, there is a word $W$ in $\delta'(w)$ of length $n$.  By Lemma~\ref{deltaCircular}, the circular word $\circular{W}$ is square-free.  Since all words in the set $\delta'(x)$ are nearly extremal square-free for every $x\in V(D)$, it follows that the circular word $\circular{W}$ is extremal square-free.

So we may assume that $n<470$.   In this case, we found an extremal square-free circular word of length $n$ by computer search.
\end{proof}

Note that the proof of Theorem~\ref{LongCircular} presented in this section can be adapted to provide yet another alternate proof of Theorem~\ref{CurrieTernary}.  (Note that while Theorem~\ref{CurrieTernary} was used in the proof of Lemma~\ref{EvenLengthWalks}, an inductive argument could be used instead.)

\section{Conclusion}

We have completely described the attainable lengths of extremal square-free ternary words, and of extremal square-free ternary circular words.  It is well-known that the number of square-free ternary words of length $n$ grows exponentially in $n$~\cite{Brandenburg1983}; currently, the best known bounds on the growth rate are due to Shur~\cite{Shur2012}.  It is also known that the number of square-free ternary circular words of length $n$ grows exponentially in $n$~\cite{Shur2010}.  Using these results together with the results of this paper, one can show that both the number of extremal square-free ternary words of length $n$, and the number of extremal square-free ternary circular words of length $n$, grow exponentially in $n$.

Surprisingly, over larger alphabets, it appears that there are no extremal square-free words.  The following is a minor variant of a conjecture of Grytczuk et al.~\cite[Conjecture~2]{Grytczuk2019}.

\begin{conjecture}\label{GrytczukConjecture}
Let $\Sigma$ be a fixed alphabet of size at least $4$.  Then there are no extremal square-free words over $\Sigma$.
\end{conjecture}

In other words, Conjecture~\ref{GrytczukConjecture} says that every square-free word over an alphabet $\Sigma$ of size greater than $3$ has at least one square-free extension over $\Sigma$.  While it would be most interesting to establish Conjecture~\ref{GrytczukConjecture} in the case $|\Sigma|=4$, establishing the conjecture for larger alphabets would also be interesting.

Finally, we note that Harju~\cite{Harju2019} recently introduced the related notion of \emph{irreducibly square-free words}; these are square-free words in which the removal of any interior letter produces a square.  In particular, Harju demonstrated that there are irreducibly square-free ternary words of every sufficiently large length, just as we have shown for extremal square-free ternary words.  However, the situation appears to be quite different over larger alphabets.  Let $n\geq 4$, and let $\Sigma_n=\{\tt{1},\tt{2},\ldots,\tt{n}\}$.  For every $3\leq k\leq n$, define $u_k=\tt{k2k}$.  Then define
\[
w_n=\tt{121}u_3\tt{121}u_4\cdots\tt{121}u_n\tt{121}=\tt{121323121424}\cdots\tt{121n2n121}.
\]
It is straightforward to verify that $w_n$ is irreducibly square-free.  So there are irreducibly square-free words over any fixed alphabet, while Conjecture~\ref{GrytczukConjecture} suggests that this is not the case for extremal square-free words.

\end{document}